\documentclass{amsart}
\usepackage{amsfonts}

\usepackage{amscd}
\usepackage{thmdefs}

\input{tcilatex}
\theoremstyle{definition}
\theoremstyle{remark}
\numberwithin{equation}{section}

\input tcilatex

\begin{document}
\title[Generalized means]{Some measure-theoretic properties of generalized means}
\author{Irina Navrotskaya}
\address{Department of Mathematics, University of Pittsburgh, Pittsburgh, PA 15260}
\email{IRN6@pitt.edu.}
\author{Patrick J. Rabier}
\address{Department of Mathematics, University of Pittsburgh, Pittsburgh, PA 15260}
\email{rabier@imap.pitt.edu}
\maketitle

\begin{abstract}
If $\Lambda $ is a measure space, $u:\Lambda ^{m}\rightarrow \Bbb{R}$ is a
given function and $N\geq m,$ the function $U(x_{1},...,x_{N})=\left( 
\begin{array}{l}
N \\ 
m
\end{array}
\right) ^{-1}\sum_{1\leq i_{1}<\cdots <i_{m}\leq
N}u(x_{i_{1}},...,x_{i_{m}}) $ is called the generalized $N$-mean with
kernel $u,$ a terminology borrowed from $U$-statistics. Physical potentials
for systems of particles are also defined by generalized means.

This paper investigates whether various measure-theoretic concepts for
generalized $N$-means are equivalent to the analogous concepts for their
kernels: a.e. convergence of sequences, measurability, essential boundedness
and integrability with respect to absolutely continuous probability
measures. The answer is often, but not always, positive. This information is
crucial in some problems addressing the existence of generalized means
satisfying given conditions, such as the classical Inverse Problem of
statistical physics (in the canonical ensemble).
\end{abstract}

\section{Introduction\label{intro}}

Let $\Lambda $ be a set and let $g:\Lambda ^{m}\rightarrow \Bbb{R}$ be a
given function. If $N\geq m$ is an integer, the generalized $N$-mean of $g$
is the symmetric function $U:\Lambda ^{N}\rightarrow \Bbb{R}$ defined by $%
U(x_{1},...,x_{N}):=\frac{(N-m)!}{N!}\sum g(x_{i_{1}},...,x_{i_{m}}),$ where
the sum carries over all $m$-tuples $(i_{1},...,i_{m})\in \{1,...,N\}^{m}$
such that $i_{j}\neq i_{k}$ if $j\neq k$.

If $u(x_{1},...,x_{m}):=\frac{1}{m!}\sum g(x_{\sigma (1)},...,x_{\sigma
(m)}) $ where $\sigma $ runs over the permutations of $\{1,...,m\},$ then $u$
is symmetric and

\begin{equation}
U(x_{1},...,x_{N})=\left( 
\begin{array}{l}
N \\ 
m
\end{array}
\right) ^{-1}\sum_{1\leq i_{1}<\cdots <i_{m}\leq
N}u(x_{i_{1}},...,x_{i_{m}}).  \label{1}
\end{equation}
The function $u$ is customarily called the kernel of $U$ and $m$ is the
order (or degree) of $U.$ Obviously, $u=g$ if $m=1$ but, unlike $g,$ the
kernel $u$ is always uniquely determined by $U,$ a point to which we shall
return below.

Generalized means lie at the foundation of $U$-statistics, a field
extensively studied since its introduction by Hoeffding \cite{Ho48}. The
nomenclature (generalized mean, kernel, order) is taken from it. In $U$
-statistics, the kernel $u$ is given, the variables $x_{1},...,x_{N}$ are
replaced with random variables $X_{1},...,X_{N}$ and the interest centers on
the behavior of statistically relevant quantities as $N\rightarrow \infty .$

Physical potentials for systems of $N$ particles when the interaction of any 
$m<N$ particles is taken into account are also defined by generalized means
(without the scaling factor, an immaterial difference). The variable $x_{j}$
captures the relevant information about the $j^{th}$ particle (often more
than just its space position, so that $\Lambda $ need not be euclidean
space) and the energy $u(x_{i_{1}},...,x_{i_{m}})$ depends upon the
interaction of the particles $i_{1},...,i_{m}.$ In some important questions,
notably the famous ``Inverse Problem'' of statistical physics (\cite{ChCh84}
, \cite{ChChLi84}, \cite{Na14}), the potential $U$ is not given. Instead,
the issue is precisely the existence of a potential $U$ of the form (\ref{1}%
) satisfying given conditions.

For example, when the Inverse Problem is set up in the ``canonical
ensemble'', $N$ is fixed\footnote{%
A major difficulty; in the ``grand canonical ensemble'', $N$ is free and the
problem was solved by Chayes and Chayes \cite{ChCh84}.} and $U$ should
maximize a functional $\mathcal{F}(V)$ (relative entropy) over a class of
potentials $V$ of the form (\ref{1}) characterized by various integrability
conditions. Accordingly, the set $\Lambda $ is equipped with a measure $dx$
and a generalized $N$-mean $U$ of order $m$ is still defined by (\ref{1}),
but now with equality holding only a.e. on $\Lambda ^{N}$ for the product
measure. If a maximizing sequence $U_{n}$ is shown to have some type of
limit $U,$ is $U$ a generalized $N$-mean of order $m$? In the affirmative,
what are the properties of the kernel $u$ of $U$ that can be inferred from
the properties of $U$ ?

In this paper, we provide answers to the above queries which, apparently,
cannot be found elsewhere. These answers play a key role in the recent work 
\cite{Na14} by the first author and they should have value in the broad
existence question for generalized means. The existence of extremal
potentials is essential in particle physics; see for instance \cite
{NoChAyVo07} for a variant of the Inverse Problem arising in coarse-grain
modeling, but usually assumed (if not taken for granted). More generally,
there are numerous examples borrowing from information theory when
generalized means are sought that maximize some kind of entropy. Even though
the maximization involves only a finite number of parameters in many of
these examples, it should be expected that some of the more complex models
require a measure-theoretic setting.

From now on, the measure $dx$ on $\Lambda $ is $\sigma $-finite and complete
and, to avoid trivialities, $\Lambda $ has strictly positive $dx$ measure.
The $\sigma $-finiteness assumption ensures that the product measure $%
dx^{\otimes k}$ on $\Lambda ^{k}$ and its completion $d^{k}x$ are defined
and that the classical theorems (Fubini, Tonelli) are applicable. The
terminology ``a.e.'', ``null set'' or ``co-null subset'' (complement of a
null set) always refers to the measure $d^{k}x$ or its measurable subsets.
Since $d^{k}x$ is complete and $\Lambda ^{k}$ has positive $d^{k}x$ measure,
a co-null subset of $\Lambda ^{k}$ is always measurable and nonempty.

The symmetry of $u$ or $U$ is not needed in any of our results and will not
be assumed. Accordingly, the ``generalized mean operator'' $G_{m,N}$ given
by $G_{m,N}(u):=U$ is well defined (Remark \ref{rm1}) and linear on the
space of all a.e. finite functions on $\Lambda ^{m}.$ The exposition is
confined to real-valued functions, but everything can readily be extended to
the vector-valued case.

The recovery of the kernel $u$ from the generalized mean $U$ will play a
crucial role. There are elementary ways to proceed, which may explain why no
general procedure seems to be on record. (Lenth \cite{Le83} addresses the
completely different problem of finding $u$ when $U=U_{N}$ in (\ref{1}) is
known for every $N$ but the order $m$ is unknown.) However, the problem
becomes much more subtle after noticing that the most natural formulas for
the kernel are useless for measure-theoretic purposes. Indeed, after
suitable modifications of $u$ and $U$ on null sets, it may be assumed that (%
\ref{1}) holds pointwise. Then, if $m=1,$ we get $u(x)=U(x,...,x)$ but, of
course, this does not show whether the measurability of $U$ implies the
measurability of $u.$ When $m>1,$ other simple formulas for the kernel, for
instance $u(x_{1},x_{2})=\frac{N}{2}U(x_{1},x_{2},...,x_{2})-\frac{(N-2)}{2}
U(x_{2},...,x_{2})$ when $m=2,$ suffer from similar shortcomings.

More sophisticated representations of $u$ in terms of $U$ will be needed to
prove that, indeed, $u$ is measurable whenever $U$ is measurable (Theorem 
\ref{th5}). These representations call for the explicit introduction of the
``kernel operator'' $K_{m,N}$ inverse of $G_{m,N},$ i.e., $K_{m,N}(U)=u.$ We
shall show that, when $N>m\geq 2,$ $K_{m,N}(U)$ can be recovered from $U$
and the four operators $K_{m,N-1},K_{m-1,N},G_{N-1,N}$ and $G_{m-1,N-1}$
(see (\ref{17}) and also (\ref{15}) when $m=1;$ that $K_{m,m}=I$ is
trivial). It follows that many properties of the kernels can be established
by transfinite induction on the pairs $(m,N)$ with $m\leq N,$ totally
ordered by $(m,N)<(m^{\prime },N^{\prime })$ if $m<m^{\prime }$ or $%
m=m^{\prime }$ and $N<N^{\prime }.$ This method is systematically used
throughout the paper.

There are also unexpected differences between pointwise and a.e. limits of
generalized $N$-means. Suppose once again that $m=1$ and that $%
U_{n}(x_{1},...,x_{N})=N^{-1}\sum_{i=1}^{N}u_{n}(x_{i})$ for every $%
(x_{1},...,x_{N})\in \Lambda ^{N},$ so that $u_{n}(x)=U_{n}(x,...,x).$ If
the sequence $U_{n}$ has a pointwise limit $U,$ then $u_{n}$ has the
pointwise limit $u(x)=U(x,...,x).$ This is true irrespective of whether $U$
achieves infinite values on $\Lambda ^{N}.$ In contrast, if it is only
assumed that $U_{n}$ has an a.e. limit $U,$ the sequence $u_{n}(x)$ may have
no limit for \emph{any} $x\in \Lambda ;$ see Example \ref{ex1}. Nonetheless,
we shall see in the next section that if $U$ is a.e. finite, then $u_{n}(x)$
converges for a.e. $x\in \Lambda $ (Theorem \ref{th3}). Thus, the finiteness
of the limit $U,$ irrelevant when pointwise convergence is assumed \footnote{%
If $m>1,$ the finiteness of $U$ is not entirely irrelevant to the pointwise
convergence issue, which however remains markedly different from a.e.
convergence.}, makes a crucial difference when only a.e. convergence holds.

Measurability and essential boundedness are discussed in Section \ref
{measurability}. Section \ref{integrability} is devoted to more delicate
integrability issues. The general problem is as follows: A symmetric
probability density $P$ on $\Lambda ^{N}$ induces a natural symmetric
probability density $P_{(m)}$ on $\Lambda ^{m}$ with $m\leq N$ upon
integrating $P$ with respect to any set of $N-m$ variables. Is it true that
a generalized $N$-mean $U$ of order $m$ is in $L^{r}(\Lambda ^{N};Pd^{N}x)$
if and only if its kernel $u$ is in $L^{r}(\Lambda ^{m};P_{(m)}d^{m}x)$ ?

Assuming $P>0$ a.e., the answer is positive when $r=\infty ,$ but the
necessity may be false if $r<\infty $ (Example \ref{ex3}). This has
immediate and important consequences in existence questions. In a nutshell,
the problem of finding a generalized $N$-mean in $L^{r}(\Lambda
^{N};Pd^{N}x) $ is not always reducible to the problem of finding its kernel
in $L^{r}(\Lambda ^{m};P_{(m)}d^{m}x)$ and the former may have solutions
when the latter does not.

The next step is to investigate whether conditions on $P$ ensure that the
above discrepancies do not occur. Such a condition is given in Theorem \ref
{th12}. It always holds in some arbitrarily small perturbations of any
symmetric probability density (Theorem \ref{th14}) and, when $\Lambda $ has
finite $dx$ measure, it is even generic (in the sense of Baire category)
among bounded symmetric probability densities (Theorem \ref{th15}). This
supports the idea that, while probably not necessary, the condition in
question is sharp. Yet, it is instructive that it fails when $Pd^{N}x$ is
the probability that $N$ particles have given coordinates, under the natural
assumption that $P=0$ when two coordinates are equal. If so, the existence
of potentials $U=G_{m,N}(u)\in L^{r}(\Lambda ^{N};Pd^{N}x)$ with $u\notin
L^{r}(\Lambda ^{m};P_{(m)}d^{m}x)$ cannot be ruled out.

Our approach also provides good guidelines for the treatment of the more
general problem when (\ref{1}) is a weighted mean, but since there are
significant new technicalities, exceptional cases, etc., this problem is not
discussed.

\section{Almost everywhere convergence\label{convergence}}

In this section, we prove that a sequence $U_{n}$ of generalized $N$-means
of order $m$ has an a.e. finite limit $U$ if and only if the corresponding
sequence of kernels $u_{n}$ has an a.e. finite limit $u.$ The next example
shows that the ``only if'' part is false if $U$ is infinite.

\begin{example}
\label{ex1} With $\Lambda =[0,1]$ and $dx$ the Lebesgue measure, let $f_{n}$
denote the well-known sequence of characteristic functions of subintervals $%
J_{n}$ with $|J_{n}|\rightarrow 0,$ such that, when $x$ is fixed, $f_{n}(x)$
assumes both the values $0$ and $1$ for arbitrarily large $n$ (\cite[p. 94]
{Ha74}). If $u_{n}:=2n(1-f_{n}),$ then $u_{n}=0$ on $J_{n}$ and $u_{n}=2n$
otherwise. Also, $u_{n}(x)$ has no limit for any $x$ since $u_{n}(x)$
assumes both the values $0$ and $2n$ for arbitrarily large $n.$ On the other
hand, $U_{n}:=G_{1,2}(u_{n})$ is $0$ on $J_{n}\times J_{n}$ and either $n$
or $2n$ otherwise. As a result, $U_{n}$ tends to $\infty $ off the diagonal
of $[0,1]^{2}$ -a set of $d^{2}x$ measure $0$- since $(x_{1},x_{2})\notin
J_{n}\times J_{n}$ if $x_{1}\neq x_{2}$ and $n$ is large enough. Thus, $%
U_{n} $ has an a.e. limit but $u_{n}$ does not.
\end{example}

If $U_{n}=G_{m,N}(u_{n}),$ then (a) $u_{n}$ and $U_{n}$ can be modified on $%
n $-independent null sets of $\Lambda ^{m}$ and $\Lambda ^{N},$
respectively, in such a way that $u_{n}$ is everywhere finite and that $%
U_{n} $ is given by (\ref{1}) for every $(x_{1},...,x_{N})\in \Lambda ^{N}$
(so that $U_{n}$ is everywhere finite; see also Remark \ref{rm1} below). On
the other hand, (b) if $U_{n}\rightarrow U$ a.e., $U_{n}$ and $U$ can be
modified on null sets independent of $n$ in such a way that convergence
holds everywhere. However, (a) and (b) \emph{cannot be achieved
simultaneously}. Indeed, a modification of the left-hand side of (\ref{1})
on a null set of $\Lambda ^{N}$ need not preserve the sum structure of the
right-hand side, whereas such a modification may be needed to ensure that $%
U_{n}\rightarrow U$ pointwise. Thus, it is generally not possible to assume
both pointwise convergence and pointwise sum structure of $U_{n}.$ In fact,
if this were always possible, the sequence $u_{n}(x)$ of Example \ref{ex1}
would have the a.e. limit $U(x,x),$ which is false for every $x.$

We begin with a simple lemma.

\begin{lemma}
\label{lm1}Let $1\leq k\leq N$ be integers and let $T_{k}$ be a co-null set
of $\Lambda ^{k}.$ Then, the set $T_{N}:=\{(x_{1},...,x_{N})\in \Lambda
^{N}:(x_{j_{1}},...,x_{j_{k}})\in T_{k}$ for every $1\leq j_{1}<\cdots
<j_{k}\leq N\}$ is co-null in $\Lambda ^{N}.$
\end{lemma}

\begin{proof}
As a preamble, note that the product $S\times T$ of two co-null sets $S$ and 
$T$ in $\Lambda ^{j}$ and $\Lambda ^{\ell },$ respectively, is co-null in
the product $\Lambda ^{j+\ell }$ since its complement $\left( (\Lambda
^{j}\backslash S)\times \Lambda ^{\ell }\right) \cup \left( \Lambda
^{j}\times (\Lambda ^{\ell }\backslash T)\right) $ is a null set in $\Lambda
^{j+\ell }.$ This feature is immediately extended to any finite product of
co-null sets. Now, $T_{N}$ is the intersection of the sets $%
\{(x_{1},...,x_{N})\in \Lambda ^{N}:(x_{j_{1}},...,x_{j_{k}})\in T_{k}\}$
for fixed $1\leq j_{1}<\cdots <j_{k}\leq N.$ Such a set is the product of $%
N-k$ copies of $\Lambda $ and one copy of $T_{k}$ and therefore co-null in $%
\Lambda ^{N}$ from the above. Thus, $T_{N}$ is co-null.
\end{proof}

\begin{remark}
\label{rm1}It follows at once from Lemma \ref{lm1} that an a.e. modification
of $u$ creates only an a.e. modification of $G_{m,N}(u).$
\end{remark}

The next lemma is the special case $m=1$ of Theorem \ref{th3} below.

\begin{lemma}
\label{lm2}Let $U_{n}$ be a sequence of generalized $N$-means of order $1$
and let $u_{n}$ denote the corresponding sequence of kernels (i.e., $%
U_{n}=G_{1,N}(u_{n})$). Then, there is an a.e. finite function $U$ on $%
\Lambda ^{N}$ such that $U_{n}\rightarrow U$ a.e. if and only if there is an
a.e. finite function $u$ on $\Lambda $ such that $u_{n}\rightarrow u$ a.e.
on $\Lambda .$ If so, $U=G_{1,N}(u).$\newline
\end{lemma}

\begin{proof}
Assume first that $u$ exists. With no loss of generality, we may also assume
that $u_{n}$ and $u$ are everywhere defined and that $u_{n}\rightarrow u$
pointwise on $\Lambda .$ Then, $U_{n}\rightarrow G_{1,N}(u)$ pointwise.

Conversely, assume that $U$ exists. With no loss of generality, we may also
assume that $u_{n}$ and $U$ are everywhere finite and that $%
U_{n}(x_{1},...,x_{N})=N^{-1}\sum_{i=1}^{N}u_{n}(x_{i})\in \Bbb{R}$ for
every $(x_{1},...,x_{N})\in \Lambda ^{N}.$ By hypothesis, 
\begin{equation}
U_{n}(x_{1},...,x_{N})\rightarrow U(x_{1},...,x_{N}),\text{ }  \label{2}
\end{equation}
for a.e. $(x_{1},...,x_{N})\in \Lambda ^{N},$ but recall that it cannot be
assumed that $U_{n}\rightarrow U$ pointwise; see the comments after Example 
\ref{ex1}.

Let $E\subset \Lambda ^{N}$ denote the co-null set on which (\ref{2}) holds.
There is a co-null set $\widetilde{S}_{N-1}$ of $\Lambda ^{N-1}$ such that,
for every $(\widetilde{x}_{2},...,\widetilde{x}_{N})\in \widetilde{S}_{N-1},$
the subset $E_{\widetilde{x}_{2},...,\widetilde{x}_{N}}:=\{x\in \Lambda :(x,%
\widetilde{x}_{2},...,\widetilde{x}_{N})\in E\}$ is co-null in $\Lambda .$
This is merely a rephrasing of the fact that a.e. $(\widetilde{x}_{2},...,%
\widetilde{x}_{N})$-section of the complement of $E$ is a null set of $%
\Lambda .$

From now on, $(\widetilde{x}_{2},...,\widetilde{x}_{N})\in \widetilde{S}
_{N-1}$ is chosen once and for all. By definition of $\widetilde{S}_{N-1},$
if $x\in E_{\widetilde{x}_{2},...,\widetilde{x}_{N}},$ then 
\begin{equation}
u_{n}(x)+u_{n}(\widetilde{x}_{2})+\cdots +u_{n}(\widetilde{x}
_{N})\rightarrow NU(x,\widetilde{x}_{2},...,\widetilde{x}_{N}).  \label{3}
\end{equation}
Thus, if $(x_{1},...,x_{N})\in E_{\widetilde{x}_{2},...,\widetilde{x}
_{N}}^{N},$ (\ref{3}) implies 
\begin{equation*}
u_{n}(x_{j})+u_{n}(\widetilde{x}_{2})+\cdots +u_{n}(\widetilde{x}
_{N})\rightarrow NU(x_{j},\widetilde{x}_{2},...,\widetilde{x}_{N}),1\leq
j\leq N
\end{equation*}
and so, by addition, 
\begin{equation}
u_{n}(x_{1})+\cdots +u_{n}(x_{N})+N(u_{n}(\widetilde{x}_{2})+\cdots +u_{n}(%
\widetilde{x}_{N}))\rightarrow N\widetilde{l}(x_{1},...,x_{N}),  \label{4}
\end{equation}
where $\widetilde{l}(x_{1},...,x_{N}):=\Sigma _{j=1}^{N}U(x_{j},\widetilde{x}
_{2},...,\widetilde{x}_{N}),$ i.e., 
\begin{equation}
\widetilde{l}=NG_{1,N}(U(\cdot ,\widetilde{x}_{2},...,\widetilde{x}_{N})).
\label{5}
\end{equation}

Since both $E$ and $E_{\widetilde{x}_{2},...,\widetilde{x}_{N}}^{N}$ are
co-null in $\Lambda ^{N},$ their intersection $E\cap E_{\widetilde{x}%
_{2},...,\widetilde{x}_{N}}^{N}$ is co-null. Let $(x_{1},...,x_{N})\in E\cap
E_{\widetilde{x}_{2},...,\widetilde{x}_{N}}^{N}$ be chosen once and for all,
so that (\ref{2}) and (\ref{4}) hold. This implies that $u_{n}(\widetilde{x}%
_{2})+\cdots +u_{n}(\widetilde{x}_{N})$ has a finite limit, namely $%
\widetilde{l}(x_{1},...,x_{N})-U(x_{1},...,x_{N}).$ Then, by (\ref{3}) and (%
\ref{5}), $u_{n}(x)$ tends to the finite limit 
\begin{multline}
u(x):=  \label{6} \\
NU(x,\widetilde{x}_{2},...,\widetilde{x}_{N})+U(x_{1},...,x_{N})-NG_{1,N}(U(%
\cdot ,\widetilde{x}_{2},...,\widetilde{x}_{N}))(x_{1},...,x_{N}),
\end{multline}
for every $x\in E_{\widetilde{x}_{2},...,\widetilde{x}_{N}}.$ Thus, $%
u_{n}\rightarrow u$ a.e. on $\Lambda $ and so, by the first part of the
proof, $U=G_{1,N}(u).$
\end{proof}

\begin{theorem}
\label{th3}Let $U_{n}$ be a sequence of generalized $N$-means of order $%
m\geq 1$ and let $u_{n}$ denote the corresponding sequence of kernels (i.e., 
$U_{n}=G_{m,N}(u_{n})$). Then, there is an a.e. finite function $U$ on $%
\Lambda ^{N}$ such that $U_{n}\rightarrow U$ a.e. if and only if there is an
a.e. finite function $u$ on $\Lambda ^{m}$ such that $u_{n}\rightarrow u$
a.e. on $\Lambda ^{m}.$ If so, $U=G_{m,N}(u).$
\end{theorem}

\begin{proof}
As in the proof of Lemma \ref{lm2} when $m=1,$ the sufficiency is
straightforward. We now assume that $U$ exists. With no loss of generality,
we may also assume that $u_{n}$ and $U$ are everywhere finite and that 
\begin{equation}
U_{n}(x_{1},...,x_{N})=\left( 
\begin{array}{l}
N \\ 
m
\end{array}
\right) ^{-1}\sum_{1\leq i_{1}<\cdots <i_{m}\leq
N}u_{n}(x_{i_{1}},...,x_{i_{m}})\in \Bbb{R},  \label{7}
\end{equation}
for every $(x_{1},...,x_{N})\in \Lambda ^{N}.$ By hypothesis, 
\begin{equation}
U_{n}(x_{1},...,x_{N})\rightarrow U(x_{1},...,x_{N}),\text{ }  \label{8}
\end{equation}
for a.e. $(x_{1},...,x_{N})\in \Lambda ^{N}.$ Once again, it cannot also be
assumed that $U_{n}\rightarrow U$ pointwise.

The proof that $u_{n}$ has an a.e. finite limit $u$ will proceed by
transfinite induction on the pairs $(m,N)$ with $N\geq m,$ totally ordered
by $(m,N)<(m^{\prime },N^{\prime })$ if $m<m^{\prime }$ or if $m=m^{\prime }$
and $N<N^{\prime }.$ For the pairs $(m,m)$ (with no predecessor), this must
be proved directly, but is trivial since $u_{n}=U_{n}$ in this case. If $%
m=1, $ the result follows from Lemma \ref{lm2} irrespective of $N.$
Accordingly, we assume $m\geq 2$ and that the result is true for generalized 
$N$-means of order $m-1$ with $N\geq m-1$ (hypothesis of induction on $m,$
which is satisfied when $m=2$) and for generalized $(N-1)$-means of order $m$
with $N\geq m+1$ (hypothesis of induction on $N$) and show that it remains
true for $u_{n}$ above.

Let $E\subset \Lambda ^{N}$ denote the co-null set on which (\ref{8}) holds.
There is a co-null set $\widetilde{S}$ of $\Lambda $ such that, for every $%
\widetilde{x}_{N}\in \widetilde{S},$ the subset $E_{\widetilde{x}
_{N}}:=\{(x_{1},...,x_{N-1})\in \Lambda ^{N-1}:(x_{1},...,x_{N-1},\widetilde{%
x}_{N})\in E\}$ is co-null in $\Lambda ^{N-1}.$ Let $\widetilde{x}_{N}\in $ $%
\widetilde{S}$ be chosen once and for all. If $(x_{1},...,x_{N-1})\in E_{%
\widetilde{x}_{N}},$ then by splitting the cases when $i_{m}\leq N-1$ and $%
i_{m}=N$ in the right-hand side of (\ref{7}), we may rewrite (\ref{8}) with $%
x_{N}=\widetilde{x}_{N}$ as 
\begin{multline}
\sum_{1\leq i_{1}<\cdots <i_{m}\leq
N-1}u_{n}(x_{i_{1}},...,x_{i_{m}})+\sum_{1\leq i_{1}<\cdots <i_{m-1}\leq
N-1}u_{n}(x_{i_{1}},...,x_{i_{m-1}},\widetilde{x}_{N})\rightarrow  \label{9}
\\
\left( 
\begin{array}{l}
N \\ 
m
\end{array}
\right) U(x_{1},...,x_{N-1},\widetilde{x}_{N}).
\end{multline}

By Lemma \ref{lm1}, the set $\widetilde{\Omega }:=\{(x_{1},...,x_{N})\in
\Lambda ^{N}:(x_{j_{1}},...,x_{j_{N-1}})\in E_{\widetilde{x}_{N}}$ for every 
$1\leq j_{1}<\cdots <j_{N-1}\leq N\}$ is co-null in $\Lambda ^{N}$ and, if $%
(x_{1},...,x_{N})\in \widetilde{\Omega },$ then (\ref{9}) holds with $%
(x_{1},...,x_{N-1})$ replaced with $(x_{j_{1}},...,x_{j_{N-1}})$ and $1\leq
j_{1}<\cdots <j_{N-1}\leq N$ fixed. If so, the $m$-tuple $%
(x_{i_{1}},...,x_{i_{m}})$ in the left-hand side of (\ref{9}) becomes $%
(x_{j_{i_{1}}},...,x_{j_{i_{m}}}).$ Since $j_{1}<\cdots <j_{N-1}$ and $%
i_{1}<\cdots <i_{m},$ it follows that $j_{i_{1}}<\cdots <j_{i_{m}}.$ We do
not write down the corresponding formula (\ref{9}) for this case, because
there is a simpler (and otherwise crucial) way to formulate it without
involving sub-subscripts, as detailed below.

There are only $N$ different ways to pick $N-1$ variables $%
x_{j_{1}},...,x_{j_{N-1}}$ with increasing indices among the $N$ variables $%
x_{1},...,x_{N},$ that is, by omitting a different variable $x_{k},1\leq
k\leq N.$ For every such $k,$ the variables $x_{j_{1}},...,x_{j_{N-1}}$ are
then $x_{1},...,\widehat{x_{k}},...,x_{N}$ where, as is customary, $\widehat{
x_{k}}$ means that $x_{k}$ is omitted. Thus, as $i_{1},...,i_{m}$ run over
all the indices such that $1\leq i_{1}<\cdots <i_{m}\leq N-1,$ the
corresponding $m$-tuples $(x_{j_{i_{1}}},...,x_{j_{i_{m}}})$ run over all
the $m$-tuples $(x_{\ell _{1}},...,x_{\ell _{m}})$ in which the variable $%
x_{k}$ does not appear, that is, over the $m$-tuples $(x_{\ell
_{1}},...,x_{\ell _{m}})$ with $1\leq \ell _{1}<\cdots <\ell _{m}\leq N$ and 
$\ell _{1}\neq k,...,\ell _{m}\neq k.$

In light of the above, when $(x_{i_{1}},...,x_{i_{m}})$ is replaced with $%
(x_{j_{i_{1}}},...,x_{j_{i_{m}}})$ in (\ref{9}) and the variable $x_{k}$
does not appear, the formula (\ref{9}) takes the form 
\begin{multline}
\sum_{\left\{ 
\begin{array}{c}
1\leq \ell _{1}<\cdots <\ell _{m}\leq N \\ 
\ell _{1}\neq k,...,\ell _{m}\neq k
\end{array}
\right. }u_{n}(x_{\ell _{1}},...,x_{\ell _{m}})+  \label{10} \\
\sum_{\left\{ 
\begin{array}{c}
1\leq \ell _{1}<\cdots <\ell _{m-1}\leq N \\ 
\ell _{1}\neq k,...,\ell _{m-1}\neq k
\end{array}
\right. }u_{n}(x_{\ell _{1}},...,x_{\ell _{m-1}},\widetilde{x}
_{N})\rightarrow \left( 
\begin{array}{l}
N \\ 
m
\end{array}
\right) U(\underset{N-1\text{ variables}}{\underbrace{x_{1},..\widehat{x_{k}}
,...,x_{N}}},\widetilde{x}_{N}).
\end{multline}
Recall that (\ref{10}) holds for every $1\leq k\leq N$ and every $%
(x_{1},...,x_{N})\in \widetilde{\Omega }.$ By adding up the relations (\ref
{10}) for $k=1,...,N,$ we get (after replacing $\ell _{1},...,\ell _{m}$
with $i_{1},...,i_{m}$) 
\begin{multline}
\sum_{k=1}^{N}\sum_{\left\{ 
\begin{array}{c}
1\leq i_{1}<\cdots <i_{m}\leq N \\ 
i_{1}\neq k,...,i_{m}\neq k
\end{array}
\right. }u_{n}(x_{i_{1}},...,x_{i_{m}})+  \label{11} \\
\sum_{k=1}^{N}\sum_{\left\{ 
\begin{array}{c}
1\leq i_{1}<\cdots <i_{m-1}\leq N \\ 
i_{1}\neq k,...,i_{m-1}\neq k
\end{array}
\right. }u_{n}(x_{i_{1}},...,x_{i_{m-1}},\widetilde{x}_{N})\rightarrow
\left( 
\begin{array}{l}
N \\ 
m
\end{array}
\right) \widetilde{l}(x_{1},...,x_{N}),
\end{multline}
for every $(x_{1},...,x_{N})\in \widetilde{\Omega },$ where $\widetilde{l}
(x_{1},...,x_{N}):=\Sigma _{k=1}^{N}U(x_{1},...\widehat{x_{k}},...,x_{N},%
\widetilde{x}_{N}),$ i.e., 
\begin{equation}
\widetilde{l}=NG_{N-1,N}(U(\cdot ,\widetilde{x}_{N})).  \label{12}
\end{equation}

Let $1\leq i_{1}<\cdots <i_{m}\leq N$ be fixed. In the first double sum of (%
\ref{11}), $u_{n}(x_{i_{1}},...,x_{i_{m}})$ appears exactly once for every
index $1\leq k\leq N$ such that $k\neq i_{1},...,k\neq i_{m}.$ Evidently,
there are $N-m$ such indices $k$ and so 
\begin{multline*}
\sum_{k=1}^{N}\sum_{\left\{ 
\begin{array}{c}
1\leq i_{1}<\cdots <i_{m}\leq N \\ 
i_{1}\neq k,...,i_{m}\neq k
\end{array}
\right. }u_{n}(x_{i_{1}},...,x_{i_{m}})= \\
(N-m)\sum_{1\leq i_{1}<\cdots <i_{m}\leq
N}u_{n}(x_{i_{1}},...,x_{i_{m}})=(N-m)\left( 
\begin{array}{l}
N \\ 
m
\end{array}
\right) U_{n}(x_{1},...,x_{N}).
\end{multline*}

The second double sum in (\ref{11}) has the same structure as the first one,
with $m$ replaced with $m-1.$ Thus, 
\begin{multline*}
\sum_{k=1}^{N}\sum_{\left\{ 
\begin{array}{c}
1\leq i_{1}<\cdots <i_{m-1}\leq N \\ 
i_{1}\neq k,...,i_{m-1}\neq k
\end{array}
\right. }u_{n}(x_{i_{1}},...,x_{i_{m-1}},\widetilde{x}_{N})= \\
(N-m+1)\sum_{1\leq i_{1}<\cdots <i_{m-1}\leq
N}u_{n}(x_{i_{1}},...,x_{i_{m-1}},\widetilde{x}_{N})= \\
m\left( 
\begin{array}{l}
N \\ 
m
\end{array}
\right) G_{m-1,N}(u_{n}(\cdot ,\widetilde{x}_{N}))(x_{1},...,x_{N}).
\end{multline*}
Hence, (\ref{11}) boils down to (recall (\ref{12})) 
\begin{equation*}
(N-m)U_{n}+mG_{m-1,N}(u_{n}(\cdot ,\widetilde{x}_{N}))\rightarrow
NG_{N-1,N}(U(\cdot ,\widetilde{x}_{N}))\text{ on }\widetilde{\Omega }.
\end{equation*}
By (\ref{8}), $U_{n}\rightarrow U$ on $E.$ Thus, 
\begin{equation}
G_{m-1,N}(u_{n}(\cdot ,\widetilde{x}_{N}))\rightarrow m^{-1}\left(
NG_{N-1,N}(U(\cdot ,\widetilde{x}_{N}))-(N-m)U\right) \text{ on }\widetilde{
\Omega }\cap E.  \label{13}
\end{equation}
Since $\widetilde{\Omega }\cap E$ is co-null in $\Lambda ^{N},$ this means
that the sequence $G_{m-1,N}(u_{n}(\cdot ,\widetilde{x}_{N}))$ has an a.e.
finite limit. Therefore, the hypothesis of induction on $m$ ensures that
there is an a.e. finite function $v$ on $\Lambda ^{m-1}$ such that $%
u_{n}(\cdot ,\widetilde{x}_{N})\rightarrow v$ a.e. on $\Lambda ^{m-1}.$
Thus, $G_{m-1,N-1}(u(\cdot ,\widetilde{x}_{N}))\rightarrow $ $G_{m-1,N-1}(v)$
on some co-null subset $T_{N-1}$ of $\Lambda ^{N-1}$ and then, from (\ref{9}%
), 
\begin{equation}
G_{m,N-1}(u_{n})\rightarrow \frac{N}{N-m}U(\cdot ,\widetilde{x}_{N})-\frac{m%
}{N-m}G_{m-1,N-1}(v)\text{ on }T_{N-1}\cap E_{\widetilde{x}_{N}}.  \label{14}
\end{equation}
Since $T_{N-1}\cap E_{\widetilde{x}_{N}}$ is co-null in $\Lambda ^{N-1},$
this shows that $G_{m,N-1}(u_{n})$ has an a.e. finite limit. That $u_{n}$
has an a.e. finite limit $u$ thus follows from the hypothesis of induction
on $N$ and so, by the first part of the proof, $U=G_{m,N}(u).$
\end{proof}

The following corollary gives a short rigorous proof of an otherwise
intuitively clear property.

\begin{corollary}
\label{cor4}A generalized $N$-mean $U$ of order $m$ has a unique kernel $u$
(up to modifications on a null set).
\end{corollary}

\begin{proof}
If $u$ and $v$ are two kernels of $U,$ set $u_{n}=u$ if $n$ is odd and $%
u_{n}=v$ if $n$ is even. Then, $G_{m,N}(u_{n})=U$ a.e. for every $n.$ By
Theorem \ref{th3}, $u_{n}$ converges a.e., which implies $u=v$ a.e.
\end{proof}

Significant differences between $m=1$ and $m>1$ are worth pointing out
(without proof, for brevity). When $m=1,$ a stronger form of Lemma \ref{lm2}
holds: If $G_{1,N}(u_{n})\rightarrow U$ a.e. with $U$ finite on a subset of $%
\Lambda ^{N}$ of positive $d^{N}x$ measure, then $u_{n}\rightarrow u$ a.e.
and $u$ is finite on a subset of $\Lambda $ of positive $dx$ measure (the
converse is clearly false). There is no such improvement of Theorem \ref{th3}
if $m>1.$ It is not hard to find examples (variants of Example \ref{ex1})
when $G_{m,N}(u_{n})$ has an a.e. limit which is finite on a subset of $%
\Lambda ^{N}$ of positive $d^{N}x$ measure, but $u_{n}$ has no limit at
every point of a subset of $\Lambda $ of positive $dx$ measure.

\section{Kernel recovery and measurability\label{measurability}}

As pointed out in the Introduction, there are rather simple ways to recover $%
u$ from $G_{m,N}(u),$ but the resulting formulas are inadequate to answer
even the most basic measure-theoretic questions. Below, we describe a
recovery procedure that preserves all the relevant measure-theoretic
information.

The fact that the kernel of a generalized $N$-mean of order $m$ is unique
(Corollary \ref{cor4}) means that $G_{m,N}$ has an inverse $K_{m,N}$
(``kernel operator''), i.e., $K_{m,N}(U)=u.$ By the linearity of $G_{m,N},$
it follows that $K_{m,N}$ (defined on the space of generalized $N$-means of
order $m$) is also linear.

Explicit formulas for $K_{1,N}$ and $K_{m,m}$ will be given but, when $m\geq
2$ and $N\geq m+1,$ $K_{m,N}$ will only be defined inductively, in terms of $%
K_{m-1,N}$ and $K_{m,N-1}$ (and also $G_{m-1,N-1}$ and $G_{N-1,N}$). This
makes it possible, by transfinite induction (as in the proof of Theorem \ref
{th3}) to recover $K_{m,N}$ for arbitrary $m$ and $N\geq m.$

First, $G_{m,m}=K_{m,m}=I$ is obvious. To define $K_{1,N},$ we return to the
proof of Lemma \ref{lm2} in the case when the sequences $u_{n}$ and $U_{n}$
are constant and therefore equal to their a.e. limits $u$ and $U,$
respectively. If so, (\ref{6}) yields the kernel of $U:$ 
\begin{multline}
K_{1,N}(U)=  \label{15} \\
NU(\cdot ,\widetilde{x}_{2},...,\widetilde{x}
_{N})+U(x_{1},...,x_{N})-NG_{1,N}(U(\cdot ,\widetilde{x}_{2},...,\widetilde{x%
}_{N}))(x_{1},...,x_{N}),
\end{multline}
where $\widetilde{x}_{2},...,\widetilde{x}_{N}$ is arbitrarily chosen in
some suitable co-null subset of $\Lambda ^{N-1}$ and $(x_{1},...,x_{N})$ is
arbitrarily chosen in some suitable co-null subset of $\Lambda ^{N}.$ (Of
course, these co-null subsets depend on $U.$)

Note that (\ref{15}) immediately shows that $K_{1,N}(U)$ is measurable if $U$
is measurable since it is always possible to choose $\widetilde{x}_{2},...,%
\widetilde{x}_{N}$ such that $U(\cdot ,\widetilde{x}_{2},...,\widetilde{x}
_{N})$ is measurable (the extra term $U(x_{1},...,x_{N})-NG_{1,N}(U(\cdot ,%
\widetilde{x}_{2},...,\widetilde{x}_{N}))(x_{1},...,x_{N})$ in (\ref{15}) is
just a constant).

To define $K_{m,N}$ in terms of $K_{m-1,N}$ and $K_{m,N-1}$ when $m\geq 2$
and $N\geq m+1,$ we return to the proof of Theorem \ref{th3} when the
sequences $u_{n}$ and $U_{n}$ are constant and therefore equal to their a.e.
limits $u$ and $U,$ respectively. If so, $u_{n}=u$ and (\ref{13}) is the
a.e. equality 
\begin{equation*}
G_{m-1,N}(u(\cdot ,\widetilde{x}_{N}))=m^{-1}\left( NG_{N-1,N}(U(\cdot ,%
\widetilde{x}_{N}))-(N-m)U\right) ,\text{ }
\end{equation*}
where $\widetilde{x}_{N}$ is arbitrarily chosen in some suitable co-null
subset of $\Lambda .$ This shows that $NG_{N-1,N}(U(\cdot ,\widetilde{x}
_{N}))-(N-m)U$ is a generalized $N$-mean of order $m-1$ and that 
\begin{equation}
u(\cdot ,\widetilde{x}_{N})=m^{-1}K_{m-1,N}\left( NG_{N-1,N}(U(\cdot ,%
\widetilde{x}_{N}))-(N-m)U\right) .  \label{16}
\end{equation}
Next, (\ref{14}) is the a.e. equality (since $v=u(\cdot ,\widetilde{x}_{N})$
in (\ref{14})) 
\begin{equation*}
G_{m,N-1}(u)=\frac{N}{N-m}U(\cdot ,\widetilde{x}_{N})-\frac{m}{N-m}%
G_{m-1,N-1}(u(\cdot ,\widetilde{x}_{N}))\text{.}
\end{equation*}
This shows that $NU(\cdot ,\widetilde{x}_{N})-mG_{m-1,N-1}(u(\cdot ,%
\widetilde{x}_{N}))$ is a generalized $(N-1)$-mean of order $m$ and, by (\ref
{16}), that $K_{m,N}(U)=u$ is given by (when $2\leq m\leq N-1$) 
\begin{multline}
(N-m)K_{m,N}(U)=  \label{17} \\
K_{m,N-1}\left( NU(\cdot ,\widetilde{x}_{N})-G_{m-1,N-1}\left(
K_{m-1,N}\left( NG_{N-1,N}(U(\cdot ,\widetilde{x}_{N}))-(N-m)U\right)
\right) \right) .
\end{multline}

If an a.e. defined function on $\Lambda ^{k}$ is said to be symmetric if it
coincides a.e. with an everywhere defined symmetric function, then a
generalized $N$-mean is symmetric if and only if its kernel is symmetric.
Indeed, it is obvious that the operators $G_{m,N}$ preserve symmetry.
Conversely, since $K_{m,m}=I$ and $K_{1,N}$ (obviously) preserve symmetry, $%
K_{m,N}$ preserves symmetry by (\ref{17}) and transfinite induction.

\begin{theorem}
\label{th5}A generalized $N$-mean $U$ of order $m$ is measurable if and only
if its kernel $u$ is measurable.
\end{theorem}

\begin{proof}
It is obvious that the measurability of $u$ implies the measurability of $U.$
From now on, we assume that $U$ is measurable and prove that $u$ is
measurable.

This is trivial when $m=N$ and was already noted after (\ref{15}) when $m=1.$
If $m\geq 2$ and $N\geq m+1,$ we proceed by transfinite induction, thereby
assuming that the theorem is true if $m$ is replaced with $m-1$ and, with $m$
being held fixed, if $N$ is replaced with $N-1.$

In (\ref{17}), choose $\widetilde{x}_{N}$ so that $U(\cdot ,\widetilde{x}
_{N})$ is measurable. This is possible since $\widetilde{x}_{N}$ is
arbitrary in a co-null subset of $\Lambda .$ In particular, $%
G_{N-1,N}(U(\cdot ,\widetilde{x}_{N}))$ is measurable (recall that the
measurability of the kernel always implies the measurability of the
generalized mean) and so $NG_{N-1,N}(U(\cdot ,\widetilde{x}_{N}))-(N-m)U$ is
measurable. Since this is a generalized $N$-mean of order $m-1$ (see the
proof of (\ref{17})), the hypothesis of induction on $m$ ensures that its
kernel is measurable. It follows that $NU(\cdot ,\widetilde{x}
_{N})-G_{m-1,N-1}\left( K_{m-1,N}\left( NG_{N-1,N}(U(\cdot ,\widetilde{x}
_{N}))-(N-m)U\right) \right) $ is measurable. Since this is a generalized $%
(N-1)$-mean of order $m$ with kernel $(N-m)K_{m,N}(U)=(N-m)u$ (see once
again the proof of (\ref{17})), it follows from the hypothesis of induction
on $N$ that $(N-m)u$ is measurable. Thus, $u$ is measurable.
\end{proof}

To complement the measurability discussion, we investigate essential
boundedness. In the next theorem, $||\cdot ||_{\Lambda ^{k},\infty }$
denotes the $L^{\infty }$ norm on $\Lambda ^{k}.$

\begin{theorem}
\label{th6}(i) If $u$ is a measurable function on $\Lambda ,$ its
generalized $N$-mean $U:=G_{1,N}(u)$ is essentially bounded above (below) if
and only if $u$ is bounded above (below). Furthermore, $\limfunc{ess}\sup U=%
\limfunc{ess}\sup u$ and $\limfunc{ess}\inf U=\limfunc{ess}\inf u.$ In
particular, $u\in L^{\infty }(\Lambda ;dx)$ if and only if $U\in L^{\infty
}(\Lambda ^{N};d^{N}x)$ and, if so, $||U||_{\Lambda ^{N},\infty
}=||u||_{\Lambda ,\infty }.$\newline
(ii) Let $1\leq m\leq N$ be integers. If $u\in L^{\infty }(\Lambda
^{m};d^{m}x),$ then $G_{m,N}(u)\in L^{\infty }(\Lambda ^{N};d^{N}x)$ and $%
||G_{m,N}(u)||_{\Lambda ^{N},\infty }\leq ||u||_{\Lambda ^{m},\infty }.$
Conversely, if $U\in L^{\infty }(\Lambda ^{N};d^{N}x)$ is a generalized $N$
-mean of order $m,$ then $K_{m,N}(U)\in L^{\infty }(\Lambda ^{m};d^{m}x)$
and there is a constant $C(m,N)$ independent of $U$ such that $%
||K_{m,N}(U)||_{\Lambda ^{m},\infty }\leq C(m,N)||U||_{\Lambda ^{N},\infty
}. $
\end{theorem}

\begin{proof}
(i) If $a\in \Bbb{R}$ and $u\geq a$ on a subset $E$ of $\Lambda $ of
positive $dx$ measure, then $U\geq a$ on $E^{N},$ of positive $d^{N}x$
measure. This shows that $\limfunc{ess}\sup U\geq \limfunc{ess}\sup u.$ On
the other hand, it is obvious that $U\leq \limfunc{ess}\sup u$ a.e. on $%
\Lambda ^{N},$ so that $\limfunc{ess}\sup U\leq \limfunc{ess}\sup u.$ This
shows that $\limfunc{ess}\sup U=\limfunc{ess}\sup u.$ Likewise, $\limfunc{ess%
}\inf U=\limfunc{ess}\inf u.$

(ii) The first part (when $u\in L^{\infty }(\Lambda ^{m};d^{m}x)$) is
straightforward. Suppose now that $U\in L^{\infty }(\Lambda ^{N};d^{N}x)$ is
a generalized $N$-mean of order $m.$ If $m=1,$ the result with $C(1,N)=1$
follows from (i) and it is trivial when $m=N$ (with $C(m,m)=1$). To prove
the existence of $C(m,N)$ in general, we proceed once again by transfinite
induction. The hypothesis of induction is that $m\geq 2,N\geq m+1$ and that
the constants $C(m-1,N)$ and $C(m,N-1)$ exist.

In (\ref{17}), choose $\widetilde{x}_{N}$ such that $U(\cdot ,\widetilde{x}
_{N})$ is measurable and that $||U(\cdot ,\widetilde{x}_{N})||_{\Lambda
^{N-1},\infty }\leq ||U||_{\Lambda ^{N},\infty }$ (which holds for a.e. $%
\widetilde{x}_{N}\in \Lambda $). Then, $||G_{N-1,N}(U(\cdot ,\widetilde{x}
_{N})||_{\Lambda ^{N},\infty }\leq ||U||_{\Lambda ^{N},\infty }$ and it
routinely follows from (\ref{17}) that $C(m,N)$ may be defined by 
\begin{equation*}
C(m,N):=\frac{C(m,N-1)}{N-m}(N+(2N-m)C(m-1,N)).
\end{equation*}
\end{proof}

\begin{remark}
In the vector-valued case, part (i) of Theorem \ref{th6} does not make sense
as stated, but $||U||_{\infty ,\Lambda ^{N}}\leq ||u||_{\infty ,\Lambda
}\leq (2N+1)||U||_{\infty ,\Lambda ^{N}}$ follows at once from (\ref{1}) and
(\ref{15}).
\end{remark}

Clearly, the result in part (i) that $G_{1,N}(u)$ is essentially bounded
above (below) if the same thing is true of $u$ remains true for $G_{m,N}(u),$
but ``only if'' is false if $m>1,$ as shown in the following example.

\begin{example}
\label{ex2}Suppose that $\Lambda $ is the countably infinite union of
pairwise disjoint subsets with positive $dx$ measure\footnote{%
This is always true if $dx$ ($\sigma $-finite) is non atomic.}. With no loss
of generality, call these subsets $\Lambda _{i}$ with $i\in \Bbb{Z}$ and
define $u$ on $\Lambda ^{2}$ by setting $u=|i|+|j|$ on $\Lambda _{i}\times
\Lambda _{j}$ if $i\neq -j$ and $u=-|j|$ if $i=-j.$ Then, $u$ is measurable,
symmetric and not essentially bounded below. Yet, $U:=G_{2,3}(u)\geq 0$.
Indeed, if $U(x_{1},x_{2},x_{3})<0,$ at least one of the pairs $%
(x_{1},x_{2}),(x_{2},x_{3})$ or $(x_{1},x_{3})$ is in $\Lambda _{-j}\times
\Lambda _{j}$ for some $j\in \Bbb{Z}.$ With no loss of generality, suppose
that $(x_{1},x_{3})\in \Lambda _{-j}\times \Lambda _{j},$ so that $%
u(x_{1},x_{3})=-|j|.$ If $x_{2}\in \Lambda _{k}$ with $k\neq \pm j,$ then $%
u(x_{1},x_{2})=u(x_{2,}x_{3})=|j|+|k|,$ so that $%
U(x_{1},x_{2},x_{3})=|j|+2|k|>0,$ which is a contradiction. Thus, $k=j$ or $%
k=-j.$ If $k=j,$ then $u(x_{1},x_{2})=-|j|$ and $u(x_{2,}x_{3})=2|j|,$ so
that $U(x_{1},x_{2},x_{3})=0,$ another contradiction. Likewise, $%
U(x_{1},x_{2},x_{3})=0$ if $k=-j.$ Therefore, $U(x_{1},x_{2},x_{3})<0$ does
not happen.
\end{example}

Since $L^{\infty }(\Lambda ^{m};d^{m}x)$ is a Banach space, the next
corollary follows at once from Theorem \ref{th6} (ii).

\begin{corollary}
\label{cor7}Let $1\leq m\leq N$ be integers. Then $G_{m,N}$ is a linear
isomorphism of $L^{\infty }(\Lambda ^{m};d^{m}x)$ onto the subspace of $%
L^{\infty }(\Lambda ^{N};d^{N}x)$ of generalized $N$-means of order $m,$
with inverse $K_{m,N}.$ In particular, the space of measurable bounded
generalized $N$-means of order $m$ is a closed subspace of $L^{\infty
}(\Lambda ^{N};d^{N}x).$
\end{corollary}

As a by-product of Corollary \ref{cor7}, the standard properties of
convergence (strong, weak, weak*) for a sequence of measurable bounded
generalized $N$-means of order $m$ are equivalent to the corresponding
properties for the sequence of kernels.

When all the functions are everywhere finite and (\ref{1}) is always assumed
to hold pointwise, not just a.e., the formulas (\ref{15}) and (\ref{17})
hold pointwise and for every\emph{\ }choice of $(\widetilde{x}_{2},...,%
\widetilde{x}_{N})\in \Lambda ^{N-1}$ and $(x_{1},...,x_{N})\in \Lambda ^{N}$
in (\ref{15}), every choice of $\widetilde{x}_{N}\in \Lambda $ in (\ref{17}
). This is readily seen by revisiting the arguments justifying these
formulas. In particular, no measure on $\Lambda $ is needed for their
validity. On the other hand, if $\Lambda $ is a topological space and $dx$
is a Borel measure, the formulas show, once again by induction, that the
Borel measurability of $u$ and $U:=K_{m,N}(u)$ are equivalent (recall that
every section of a Borel measurable function is Borel measurable). This is
not true if (\ref{1}) does not hold pointwise or if the functions are not
everywhere defined and finite, because modification on a null set, even a
Borel one, need not preserve the Borel measurability of functions.

\section{Integrability\label{integrability}}

In the previous sections, we saw that the kernel $u$ of a generalized mean $%
U $ is uniquely determined by $U$ (up to modifications on a null set) and
that the measurability and essential boundedness of $U$ and $u$ are
equivalent. It is a natural question whether some integrability properties
are also transferred. However, it immediately appears that generalized $N$
-means cannot have good integrability properties with respect to $d^{N}x$
when $\Lambda $ has infinite measure, regardless of whether the kernel is
integrable. Rather than confining attention to the case when $\Lambda $ has
finite measure, we shall discuss integrability relative to $Pd^{N}x$ where $%
P $ is a probability density on $\Lambda ^{N},$ i.e., $P\in L^{1}(\Lambda
^{N};d^{N}x),P\geq 0$ and $\int_{\Lambda ^{N}}Pd^{N}x=1.$ The last condition
is convenient but inessential and can always be achieved by scaling.

On the other hand, it will be important to assume that $P>0$ (a.e.) and that 
$P$ is symmetric. That $P>0$ ensures that the measures $d^{N}x$ and $Pd^{N}x$
are absolutely continuous with respect to one another and, hence, that the
null sets are the same for both measures. The symmetry of $P$ is convenient
to formulate the problem of interest. In general, given $1\leq k<N,$ choose
indices $1\leq i_{1}<\cdots <i_{k}\leq N$ and consider the function $%
P_{i_{1},...,i_{k}}(x_{i_{1}},...,x_{i_{k}})$ obtained by integrating $P$
with respect to all the variables $x_{j}$ with $j\notin \{i_{1},...,i_{k}\}.$
By Fubini's theorem, $P_{i_{1},...,i_{k}}$ is a (marginal) probability
density on $\Lambda ^{k}.$ The main point is that if $P$ is symmetric, $%
P_{i_{1},...,i_{k}}(y_{1},...,y_{k})$ is independent of the choice of the
indices $i_{1},...,i_{k}.$ It will be called the $k$\emph{-variable reduction%
} $P_{(k)}$ of $P.$ Note that $P_{(k)}>0$ a.e. on $\Lambda ^{k},$ that $%
P_{(k)}$ is symmetric and that $P_{(\ell )}$ is the $\ell $-variable
reduction of $P_{(k)}$ when $1\leq \ell <k.$

We observe in passing that the symmetry of $P$ is sufficient, but not
necessary, for the existence of $P_{(k)},1\leq k\leq N-1.$ Nothing more is
needed later but, for expository purposes, we shall continue to assume that $%
P$ is symmetric.

The definition of generalized means shows at once that if $u\in
L^{1}(\Lambda ^{m};P_{(m)}d^{m}x),$ then $U:=G_{m,N}(u)\in L^{1}(\Lambda
^{N};Pd^{N}x)$ and $||U||_{1,Pd^{N}x}\leq ||u||_{1,P_{(m)}d^{m}x}.$
Furthermore, $\int_{\Lambda ^{N}}UPd^{N}x=\int_{\Lambda ^{m}}uP_{(m)}d^{m}x.$
Surprisingly, it is generally \emph{false} that $U\in L^{1}(\Lambda
^{N};Pd^{N}x)$ implies $u\in L^{1}(\Lambda ^{m};P_{(m)}d^{m}x),$ as shown in
Example \ref{ex3} below. In particular, contrary to intuition, the relation $%
\int_{\Lambda ^{N}}UPd^{N}x=\int_{\Lambda ^{m}}uP_{(m)}d^{m}x$ need not make
sense even though the left-hand side is defined. This is easy to overlook in
formal calculations.

More generally, if $1\leq r<\infty $ and $u\in L^{r}(\Lambda
^{m};P_{(m)}d^{m}x),$ then\footnote{%
Since $L^{\infty }(\Lambda ^{m};P_{(m)}d^{m}x)=L^{\infty }(\Lambda
^{m};d^{m}x)$ and $L^{\infty }(\Lambda ^{N};Pd^{N}x)=L^{\infty }(\Lambda
^{N};d^{N}x),$ this is also true when $r=\infty ;$ see Theorem \ref{th6}
(ii).} $G_{m,N}(u)\in L^{r}(\Lambda ^{N};Pd^{N}x)$ and $%
||G_{m,N}(u)||_{r,Pd^{N}x}\leq ||u||_{r,P_{(m)}d^{m}x},$ but $G_{m,N}(u)\in
L^{r}(\Lambda ^{N};Pd^{N}x)$ need not imply $u\in L^{r}(\Lambda
^{m};P_{(m)}d^{m}x).$

\begin{remark}
From the above, $\{G_{m,N}(u):u\in L^{r}(\Lambda
^{m};P_{(m)}d^{m}x)\}\subset \{G_{m,N}(u)\in L^{r}(\Lambda ^{N};Pd^{N}x)\}$
and equality need not hold, but we do not know whether the former space is
always dense in the latter. A notable difficulty is that truncation does not
preserve the generalized mean structure.
\end{remark}

\begin{example}
\label{ex3}With $\Lambda =(0,\infty ),$ let $dx$ denote the Lebesgue
measure. For $i,j\in \Bbb{N},$ set $I_{i}:=(i-1,i]$ and $Q_{ij}:=I_{i}\times %
I_{j}.$ We define $P$ on $(0,\infty )^{2}$ by 
\begin{equation}
P=p_{ij}\text{ on }Q_{ij},\text{ where }p_{ij}:=\left\{ 
\begin{array}{l}
(i+j)^{-2}\text{ if }|i-j|=1, \\ 
(i+j)^{-4}\text{ if }|i-j|\neq 1.
\end{array}
\right.   \label{18}
\end{equation}
Then, $P$ is symmetric and $P>0.$ The integrability of $P$ follows from $%
\sum_{i,j=1}^{\infty }p_{ij}=\sum_{|i-j|=1}p_{ij}+\sum_{|i-j|\neq 1}p_{ij}.$
The first sum is equal to $2\sum_{i=1}^{\infty }(2i+1)^{-2}<\infty $ and the
second one is majorized by $\sum_{i,j=1}^{\infty }(i+j)^{-4}=\sum_{k=2}^{%
\infty }\sum_{i+j=k}(i+j)^{-4}=\sum_{k=2}^{\infty }(k-1)k^{-4}<\infty .$
That $\int_{(0,\infty )^{2}}Pd^{2}x\neq 1$ is of course inconsequential. 
\newline
Next, let $u$ be given by 
\begin{equation}
u=2(-1)^{i}i\text{ on }I_{i}.  \label{19}
\end{equation}
Set $U=G_{1,2}(u).$ By (\ref{19}), $U=(-1)^{i}i+(-1)^{j}j$ on $Q_{ij}$ and $%
|(-1)^{i}i+(-1)^{j}j|=1$ when $|i-j|=1,$ whereas $|(-1)^{i}i+(-1)^{j}j|\leq %
i+j$ in all cases. The integrability of $|U|P,$ i.e., $U\in L^{1}((0,\infty %
)^{2};Pd^{2}x),$ follows from $\sum_{i,j=1}^{\infty %
}|(-1)^{i}i+(-1)^{j}j|p_{ij}=\sum_{|i-j|=1}p_{ij}+\sum_{|i-j|\neq %
1}|(-1)^{i}i+(-1)^{j}j|p_{ij}.$ As before, the first sum is $2\sum_{i=1}^{%
\infty }(2i+1)^{-2}<\infty $ and the second one is majorized by $%
\sum_{i,j=1}^{\infty }(i+j)^{-3}=\sum_{k=2}^{\infty }\sum_{i+j=k}(i+j)^{-3}=%
\sum_{k=1}^{\infty }(k-1)k^{-3}<\infty .$ \newline
Now, $P_{(1)}=\lambda _{i}$ on $I_{i},$ where $\lambda _{i}:=\sum_{j=1}^{%
\infty }p_{ij}.$ Assume $i\geq 2.$ By (\ref{18}), $\lambda
_{i}=(2i-1)^{-2}+(2i+1)^{-2}+\sum_{j\geq 1,|j-i|\neq 1}(i+j)^{-4}>(2i+1)^{-2}
$ and so, by (\ref{19}), $|u|P_{(1)}=2i\lambda _{i}>2i(2i+1)^{-2}$ on $I_{i}.
$ Since $\sum_{i=2}^{\infty }i(2i+1)^{-2}=\infty ,$ it follows that $u\notin
L^{1}((0,\infty );P_{(1)}dx).$
\end{example}

In the remainder of this section, we show that, in spite of Example \ref{ex3}
, a suitable extra condition on $P$ ensures that a generalized $N$-mean $U$
of order $m$ is in $L^{r}(\Lambda ^{N};Pd^{N}x)$ with $1\leq r<\infty $ if
and only if its kernel $K_{m,N}(U)$ is in $L^{r}(\Lambda
^{m};P_{(m)}d^{m}x). $ This condition is independent of the measure $dx$ and
of $m$ and $N.$ Its generality is addressed in the last two theorems.

Once again, we begin with the case $m=1.$

\begin{lemma}
\label{lm8} Suppose $N\geq 2$ and that for a.e. $(x_{2},...,x_{N})$ in some
measurable subset $A\subset \Lambda ^{N-1}$ of positive $d^{N-1}x$ measure,
there is a constant $\Gamma (x_{2},...,x_{N})>0$ such that 
\begin{equation}
P(\cdot ,x_{2},...,x_{N})\geq \Gamma (x_{2},...,x_{N})P_{(1)}\text{ a.e. on }%
\Lambda .  \label{20}
\end{equation}
Then, there is a constant $\alpha >0$ such that, for every $r\in [1,\infty )$
and every $U\in L^{r}(\Lambda ^{N};Pd^{N}x),$ the set of those $%
(x_{2},...,x_{N})\in A$ such that $U(\cdot ,x_{2},...,x_{N})\in
L^{r}(\Lambda ;P_{(1)}dx)$ and $||U(\cdot ,x_{2},...,x_{N})||_{r,P_{(1)}dx}%
\leq \alpha ^{1/r}||U||_{r,Pd^{N}x}$ is not a null set\footnote{%
We need not debate whether this set is measurable.} of $\Lambda ^{N-1}.$
\end{lemma}

\begin{proof}
The condition (\ref{20}) is equivalent to $\limfunc{ess}\inf_{x_{1}\in
\Lambda }P(x_{1},...,x_{N})P_{(1)}(x_{1})^{-1}>0$ for a.e. $%
(x_{2},...,x_{N})\in A.$ The function $P(x_{1},...,x_{N})P_{(1)}(x_{1})^{-1}$
is a.e. finite on $\Lambda ^{N}$ (because $P>0$ is in $L^{1}(\Lambda
^{N};d^{N}x)$ and $P_{(1)}>0$) and measurable. Thus, $\limfunc{ess}%
\inf_{x_{1}\in \Lambda }P(x_{1},\cdot )P_{(1)}(x_{1})^{-1}$ is a.e. finite
and measurable on $\Lambda ^{N-1}$ (see e.g. \cite[p. 271 and p. 275]{Jo93}%
). As a result, upon replacing $\Gamma $ with $\limfunc{ess}\inf_{x_{1}\in
\Lambda }P(x_{1},\cdot )P_{(1)}(x_{1})^{-1}\geq \Gamma ,$ it is not
restrictive to assume that $\Gamma $ is measurable on $A.$ If so, the set $%
A_{\varepsilon }:=\{(x_{2},...,x_{N})\in A:\Gamma
(x_{2},...,x_{N})>\varepsilon \}$ has positive $d^{N-1}x$ measure if $%
\varepsilon >0$ is small enough and, by (\ref{20}), we infer that for a.e. $%
(x_{2},...,x_{N})\in A_{\varepsilon },$%
\begin{equation}
P(\cdot ,x_{2},...,x_{N})\geq \varepsilon P_{(1)}\text{ a.e. on }\Lambda .
\label{21}
\end{equation}
Furthermore, the $d^{N-1}x$ measure $|A_{\varepsilon }|$ of $A_{\varepsilon }
$ is finite. Indeed, by integrating (\ref{21}) on $\Lambda \times %
A_{\varepsilon },$ we get $1\geq \int_{\Lambda \times A_{\varepsilon
}}Pd^{N}x\geq \varepsilon |A_{\varepsilon }|,$ so that $|A_{\varepsilon }|%
\leq \varepsilon ^{-1}.$ We claim that the theorem holds with $\alpha
:=|A_{\varepsilon }|^{-1}\varepsilon ^{-1}\geq 1.$

Let $U\in L^{r}(\Lambda ^{N};Pd^{N}x)$ be given, where $r\in [1,\infty ).$
Since $P>0,$ $U$ is measurable. Hence, $U(\cdot ,x_{2},...,x_{N})$ and $%
|U(\cdot ,x_{2},...,x_{N})|^{r}P_{(1)}$ are measurable on $\Lambda $ for
a.e. $(x_{2},...,x_{N})\in \Lambda ^{N-1}.$ By Fubini's theorem, $|U(\cdot
,x_{2},...,x_{N})|^{r}P(\cdot ,x_{2},...,x_{N})\in L^{1}(\Lambda ;dx)$ for
a.e. $(x_{2},...,x_{N})\in \Lambda ^{N-1}$ and so, by (\ref{21}), 
\begin{equation*}
\int_{\Lambda }|U(x,x_{2},...,x_{N})|^{r}P_{(1)}(x)dx\leq \varepsilon
^{-1}\int_{\Lambda }|U(x,x_{2},...,x_{N})|^{r}P(x,x_{2},...,x_{N})dx<\infty ,
\end{equation*}
for a.e. $(x_{2},...,x_{N})\in A_{\varepsilon }.$ This shows that $U(\cdot
,x_{2},...,x_{N})\in L^{r}(\Lambda ;P_{(1)}dx)$ for a.e. $%
(x_{2},...,x_{N})\in A_{\varepsilon }.$

By contradiction, suppose now that the inequality $||U(\cdot
,x_{2},...,x_{N})||_{r,P_{(1)}dx}\leq $\linebreak $|A_{\varepsilon
}|^{-1/r}\varepsilon ^{-1/r}||U||_{r,Pd^{N}x}$ holds only for $%
(x_{2},...,x_{N})$ in a null set of $A_{\varepsilon },$ i.e., 
\begin{equation*}
|A_{\varepsilon }|^{-1}||U||_{r,Pd^{N}x}^{r}<\varepsilon \int_{\Lambda
}|U(x,x_{2},...,x_{N})|^{r}P_{1}(x)dx,
\end{equation*}
for a.e. $(x_{2},...,x_{N})\in A_{\varepsilon }.$ Together with (\ref{21}),
this yields 
\begin{equation*}
|A_{\varepsilon }|^{-1}||U||_{r,Pd^{N}x}^{r}<\int_{\Lambda
}|U(x,x_{2},...,x_{N})|^{r}P(x,x_{2},...,x_{N})dx,
\end{equation*}
for a.e. $(x_{2},...,x_{N})\in A_{\varepsilon }.$ Then, upon integrating on $%
A_{\varepsilon },$ we find $||U||_{r,Pd^{N}x}^{r}<\int_{\Lambda \times
A_{\varepsilon }}|U|^{r}Pd^{N}x\leq ||U||_{r,Pd^{N}x}^{r},$ which is absurd.
\end{proof}

We now solve the integrability question when $m=1.$

\begin{lemma}
\label{lm9} If $N\geq 2,$ suppose that for a.e. $(x_{2},...,x_{N})$ in some
subset $A\subset \Lambda ^{N-1}$ of positive $d^{N-1}x$ measure, there is a
constant $\Gamma (x_{2},...,x_{N})>0$ such that 
\begin{equation}
P(\cdot ,x_{2},...,x_{N})\geq \Gamma (x_{2},...,x_{N})P_{(1)}\text{ a.e. on }%
\Lambda .  \label{22}
\end{equation}
Assume $1\leq r<\infty $ and that $U=G_{1,N}(u)$ for some a.e. finite
function $u$ on $\Lambda .$ Then, $U\in L^{r}(\Lambda ^{N};Pd^{N}x)$ if and
only if $u\in L^{r}(\Lambda ;P_{(1)}dx).$ If so, $||U||_{r,Pd^{N}x}\leq $ $%
||u||_{r,P_{(1)}dx}\leq C_{r}(N,P)||U||_{r,Pd^{N}x}$ where $C_{r}(N,P)>0$ is
a constant independent of $U$ and of $u.$
\end{lemma}

\begin{proof}
Since the sufficiency was already justified before Example \ref{ex3}, we
only address the necessity and, with no loss of generality, assume $N\geq 2.$
Suppose $U:=G_{1,N}(u)\in L^{r}(\Lambda ^{N};Pd^{N}x).$ Once again, since $%
P>0,$ $U$ is measurable and so, by Theorem \ref{th5}, $u$ is measurable.

Now, recall the formula (\ref{15}) for $u=K_{1,N}(U),$ in which $(\widetilde{%
x}_{2},...,\widetilde{x}_{N})$ is arbitrarily chosen in a co-null set $%
\widetilde{S}_{N-1}$ of $\Lambda ^{N-1}$ and $(x_{1},...,x_{N})$ is
arbitrarily chosen in some co-null set of $\Lambda ^{N}.$ For convenience,
rewrite this formula as 
\begin{equation}
u(x)=NU(x,\widetilde{x}_{2},...,\widetilde{x}_{N})+\widetilde{c},  \label{23}
\end{equation}
where $\widetilde{c}:=U(x_{1},...,x_{N})-NG_{1,N}(U(\cdot ,\widetilde{x}%
_{2},...,\widetilde{x}_{N}))(x_{1},...,x_{N})$ is a constant (i.e.,
independent of $x$). By (\ref{23}), $\widetilde{c}$ is uniquely determined
by $u$ and by $(\widetilde{x}_{2},...,\widetilde{x}_{N}),$ which shows that $%
\widetilde{c}$ is also independent of $(x_{1},...,x_{N}).$ Since $P$ is a
probability density on $\Lambda ^{N},$ nonnegative constant functions on $%
\Lambda ^{N}$ are equal to their $L^{r}(\Lambda ^{N};Pd^{N}x)$ norm for
every $r.$ It follows that ($NG_{1,N}(U(\cdot ,\widetilde{x}_{2},...,%
\widetilde{x}_{N}))\in L^{r}(\Lambda ^{N};Pd^{N}x)$ and) 
\begin{equation}
|\widetilde{c}|\leq ||U||_{r,Pd^{N}x}+||NG_{1,N}(U(\cdot ,\widetilde{x}
_{2},...,\widetilde{x}_{N}))||_{r,Pd^{N}x}.  \label{24}
\end{equation}

We now evaluate $||NG_{1,N}(U(\cdot ,\widetilde{x}_{2},...,\widetilde{x}
_{N}))||_{r,Pd^{N}x}$ after making a suitable choice of $(\widetilde{x}
_{2},...,\widetilde{x}_{N}).$ By Lemma \ref{lm8} and the completeness of $%
d^{N-1}x,$ there is $\alpha >0$ (independent of $U$ and $r$) such that the
subset of those points $(x_{1},...,x_{N})\in A$ such that $U(\cdot
,x_{2},...,x_{N})\in L^{r}(\Lambda ;P_{(1)}dx)$ and $||U(\cdot
,x_{2},...,x_{N})||_{r,P_{(1)}dx}\leq \alpha ^{1/r}||U||_{r,Pd^{N}x}$ is not
contained in the null set $\Lambda ^{N-1}\backslash \widetilde{S}_{N-1}.$
Hence, there is $(\widetilde{x}_{2},...,\widetilde{x}_{N})\in \widetilde{S}
_{N-1}$ such that $U(\cdot ,\widetilde{x}_{2},...,\widetilde{x}_{N})\in
L^{r}(\Lambda ;P_{(1)}dx)$ and that 
\begin{equation}
||U(\cdot ,\widetilde{x}_{2},...,\widetilde{x}_{N})||_{r,P_{(1)}dx}\leq
\alpha ^{1/r}||U||_{r,Pd^{N}x}.  \label{25}
\end{equation}

By definition of $G_{1,N},$%
\begin{equation}
NG_{1,N}(U(\cdot ,\widetilde{x}_{2},...,\widetilde{x}
_{N}))(x_{1},...,x_{N})=\sum_{j=1}^{N}U(x_{j},\widetilde{x}_{2},...,%
\widetilde{x}_{N}).  \label{26}
\end{equation}
We claim that, when $j$ is fixed, $U(x_{j},\widetilde{x}_{2},...,\widetilde{x%
}_{N})$ is in $L^{r}(\Lambda ^{N};Pd^{N}x)$ when viewed as a function of the 
$N$ variables $x_{1},...,x_{N}.$ Indeed, by the symmetry of $P$ and (\ref{25}%
), 
\begin{multline*}
\int_{\Lambda ^{N}}|U(x_{j},\widetilde{x}_{2},...,\widetilde{x}%
_{N})|^{r}P(x_{1},...,x_{N})dx_{1}\cdots dx_{N}= \\
\int_{\Lambda }|U(x_{j},\widetilde{x}_{2},...,\widetilde{x}
_{N})|^{r}P_{(1)}(x_{j})dx\leq \alpha ||U||_{r,Pd^{N}x}^{r}.
\end{multline*}
Thus, by (\ref{26}), $NG_{1,N}(U(\cdot ,\widetilde{x}_{2},...,\widetilde{x}%
_{N}))$ is the sum of $N$ functions of $L^{r}(\Lambda ^{N};Pd^{N}x)$ with
norms bounded by $\alpha ^{1/r}||U||_{r,Pd^{N}x}.$ Hence, $%
||NG_{1,N}(U(\cdot ,\widetilde{x}_{2},...,\widetilde{x}_{N}))||_{r,Pd^{N}x}%
\leq N\alpha ^{1/r}||U||_{r,Pd^{N}x}$ and so, $|\widetilde{c}|\leq
(1+N\alpha ^{1/r})||U||_{r,Pd^{N}x}$ by (\ref{24}). Then, by (\ref{23}) and (%
\ref{25}), $u\in L^{r}(\Lambda ;P_{(1)}dx)$ and $||u||_{r,P_{(1)}dx}\leq
C_{r}(N,P)||U||_{r,Pd^{N}x}$ where $C_{r}(N,P):=2N\alpha ^{1/r}+1.$
\end{proof}

When $m\geq 2,$ we first need a lemma similar to Lemma \ref{lm8}.

\begin{lemma}
\label{lm10} Suppose $N\geq 2$ and that for a.e. $x_{N}$ in some subset $%
B\subset \Lambda $ of positive $dx$ measure, there is a constant $\gamma
(x_{N})>0$ such that $P(\cdot ,x_{N})\geq \gamma (x_{N})P_{(N-1)}$ a.e. on $%
\Lambda ^{N-1}.$ Then, there is a constant $\beta >0$ such that, given $U\in
L^{r}(\Lambda ^{N};Pd^{N}x)$ with $r\in [1,\infty ),$ the set of those $%
x_{N}\in B$ such that $U(\cdot ,x_{N})\in L^{r}(\Lambda
^{N-1};P_{(N-1)}d^{N-1}x)$ and $||U(\cdot ,x_{N})||_{r,P_{(N-1)}d^{N-1}x}%
\leq \beta ^{1/r}||U||_{r,Pd^{N}x}$ is not a null set of $\Lambda .$
\end{lemma}

\begin{proof}
Modify the proof of Lemma \ref{lm8} in the obvious way.
\end{proof}

In Lemma \ref{lm10}, $P(x_{1},...,x_{N})\geq \gamma
(x_{N})P_{(N-1)}(x_{1},...,x_{N-1})$ holds on a co-null subset of $\Lambda
^{N-1}\times B,$ but this is not immediately obvious (a union of sections of 
$d^{N-1}x$ measure $0$ has $d^{N}x$ measure $0$ \emph{only} if it is $d^{N}x$
measurable, which is not generally true):

\begin{remark}
\label{rm2}The condition $P(\cdot ,x_{N})\geq \gamma (x_{N})P_{(N-1)}$ a.e.
on $\Lambda ^{N-1}$ for a.e. $x_{N}\in B$ implies the same condition with $%
\gamma $ replaced with $\widetilde{\gamma }\geq \gamma $ and $\widetilde{%
\gamma }$ measurable on $B$ (see the proof of Lemma \ref{lm8}). Since $%
P-P_{(N-1)}\otimes \widetilde{\gamma }$ is measurable on $\Lambda ^{N-1}%
\times B,$ the subset of $\Lambda ^{N-1}\times B$ where $P\geq %
P_{(N-1)}\otimes \widetilde{\gamma }$ is measurable and, by a
straightforward contradiction argument, it is co-null in $\Lambda ^{N-1}%
\times B.$ Since $\widetilde{\gamma }\geq \gamma ,$ it follows that $%
P(x_{1},...,x_{N})\geq \gamma (x_{N})P_{(N-1)}(x_{1},...,x_{N-1})$ a.e. on $%
\Lambda ^{N-1}\times B.$
\end{remark}

The next lemma is elementary, but crucial to the induction procedure.

\begin{lemma}
\label{lm11}Suppose $N\geq 3$ and that for a.e. $x_{N}$ in some subset $%
B\subset \Lambda $ of positive $dx$ measure, there is a constant $\gamma
(x_{N})>0$ such that $P(\cdot ,x_{N})\geq \gamma (x_{N})P_{(N-1)}$ a.e. on $%
\Lambda ^{N-1}.$ Then, for a.e. $x_{N-1}\in B,$ we have $P_{(N-1)}(\cdot %
,x_{N-1})\geq \gamma (x_{N-1})P_{(N-2)}$ a.e. on $\Lambda ^{N-2}.$
\end{lemma}

\begin{proof}
The short version of the proof goes as follows: For every $x_{N}\in B$ such
that $P(\cdot ,x_{N})\geq \gamma (x_{N})P_{(N-1)},$ integrate this
inequality with respect to $x_{1}$ to get $P_{(N-1)}(x_{2},...,,x_{N})\geq
\gamma (x_{N})P_{(N-2)}(x_{2},...,x_{N-1})$ and change $x_{j}$ into $x_{j-1}$
for $2\leq j\leq N.$ It is easy to justify this procedure -perhaps for $%
x_{N} $ in a smaller co--null subset of $B$- based on Fubini's theorem and
the fact that a.e. section of a co-null subset is co-null. We skip the minor
details.
\end{proof}

\begin{theorem}
\label{th12}If $N\geq 2,$ suppose that for a.e. $x_{N}$ in some subset $%
B\subset \Lambda $ of positive $dx$ measure, there is a constant $\gamma
(x_{N})>0$ such that 
\begin{equation}
P(\cdot ,x_{N})\geq \gamma (x_{N})P_{(N-1)}\text{ a.e. on }\Lambda ^{N-1}.
\label{27}
\end{equation}
Let $1\leq m\leq N$ be integers and let $1\leq r<\infty .$ \newline
(i) If $u\in L^{r}(\Lambda ^{m};P_{(m)}d^{m}x),$ then $G_{m,N}(u)\in
L^{r}(\Lambda ^{N};Pd^{N}x)$ and $||G_{m,N}(u)||_{r,Pd^{N}x}\leq %
||u||_{r,P_{(m)}d^{m}x}.$ \newline
(ii) Conversely, if $U\in L^{r}(\Lambda ^{N};Pd^{N}x)$ is a generalized $N$
-mean of order $m,$ then $K_{m,N}(U)\in L^{r}(\Lambda ^{m};P_{(m)}d^{m}x)$
and there is a constant $C(m,N,P)$ independent of $U$ such that $%
||K_{m,N}(U)||_{r,P_{(m)}d^{m}x}\leq C(m,N,P)||U||_{r,Pd^{N}x}.$
\end{theorem}

\begin{proof}
(i) As was observed earlier, this is straightforward (and (\ref{27}) is not
needed).

(ii) Since $K_{m,m}=I,$ there is nothing to prove when $m=N$ (and $%
C_{r}(m,m)=1$). Accordingly, we henceforth assume $N\geq m+1.$ If $m=1,$ the
result follows from Lemma \ref{lm9}, with $C_{r}(1,N,P)=C_{r}(N,P)$ from
that lemma. Indeed, by inductive application of Lemma \ref{lm11} and Remark 
\ref{rm2} if $N\geq 3,$ or directly from (\ref{27}) if $N=2,$ the condition (%
\ref{22}) holds with $A=B^{N-1}$ and $\Gamma (x_{2},...,x_{N}):=\gamma
(x_{N})\cdots \gamma (x_{2}).$

From now on, $m\geq 2$ (and $N\geq m+1,$ so that $N\geq 3$). By transfinite
induction, we assume that the result is true if $m$ is replaced with $m-1$
and that it is also true if $N$ is replaced with $N-1.$ Since $N$ cannot be
replaced with $N-1$ without changing $P,$ a clarification is in order: What
is actually assumed is that if $N$ is replaced with $N-1$ and if $Q>0$ is a
symmetric probability density on $\Lambda ^{N-1}$ satisfying (\ref{27}) with 
$N$ replaced with $N-1,$ there is a constant $C_{r}(m,N-1,Q)>0$ independent
of the generalized $(N-1)$-mean $V\in L^{r}(\Lambda ^{N-1};Qd^{N-1}x)$ of
order $m,$ such that $K_{m,N-1}(V)\in L^{r}(\Lambda ^{m};Q_{(m)}d^{m}x)$ and
that $||K_{m,N}(V)||_{r,Q_{(m)}d^{m}x}\leq
C_{r}(m,N-1,Q)||V||_{r,Qd^{N-1}x}. $

In (\ref{17}), $\widetilde{x}_{N}$ is arbitrary in some co-null subset of $%
\Lambda $ and so, by Lemma \ref{lm10}, we may assume that $U(\cdot ,%
\widetilde{x}_{N})\in L^{r}(\Lambda ^{N-1};P_{(N-1)}d^{N-1}x)$ and that $%
||U(\cdot ,\widetilde{x}_{N})||_{r,P_{(N-1)}d^{N-1}x}\leq \beta
^{1/r}||U||_{r,Pd^{N}x},$ where $\beta >0$ is independent of $U.$ With such
a choice of $\widetilde{x}_{N},$ it follows that the generalized $N$-mean of
order $m-1$ (see the proof of (\ref{17})) $NG_{N-1,N}(U(\cdot ,\widetilde{x}
_{N}))-(N-m)U$ is in $L^{r}(\Lambda ^{N};Pd^{N}x)$ and, recalling (i), that 
\begin{multline}
||NG_{N-1,N}(U(\cdot ,\widetilde{x}_{N}))-(N-m)U||_{r,Pd^{N}x}\leq
\label{28} \\
\left( N\beta ^{1/r}+(N-m)\right) ||U||_{r,Pd^{N}x}.
\end{multline}

Thus, we can now use the hypothesis of induction on $m$ to infer
that\linebreak $K_{m-1,N}\left( NG_{N-1,N}(U(\cdot ,\widetilde{x}
_{N}))-(N-m)U\right) \in L^{r}(\Lambda ^{m-1};P_{(m-1)}d^{m-1}x)$ and that
the inequality $||K_{m-1,N}\left( NG_{N-1,N}(U(\cdot ,\widetilde{x}%
_{N}))-(N-m)U\right) ||_{r,P_{(m-1)}d^{m-1}x}\leq $\linebreak $%
C_{r}(m-1,N,P)||NG_{N-1,N}(U(\cdot ,\widetilde{x}_{N}))-(N-m)U||_{r,Pd^{N}x}$
holds. Then, by (i) and (\ref{28}), $V:=NU(\cdot ,\widetilde{x}
_{N})-G_{m-1,N-1}\left( K_{m-1,N}\left( NG_{N-1,N}(U(\cdot ,\widetilde{x}
_{N}))-(N-m)U\right) \right) \in L^{r}(\Lambda ^{N-1};P_{(N-1)}d^{N-1}x)$
and 
\begin{multline}
||V||_{r,P_{(N-1)}d^{N-1}x}\leq  \label{29} \\
\left( N\beta ^{1/r}+C_{r}(m-1,N,P)\left( N\beta ^{1/r}+(N-m)\right) \right)
||U||_{r,Pd^{N}x}.
\end{multline}

From the proof of (\ref{17}), $V$ above is a generalized $(N-1)$-mean of
order $m$ with kernel $K_{m,N-1}(V)=K_{m,N}(U).$ By Lemma \ref{lm11}, $%
Q=P_{(N-1)}$ satisfies (\ref{27}) and $Q_{(m)}=P_{(m)}.$ Thus, by the
hypothesis of induction on $N,$ it follows that $K_{m,N}(U)=K_{m,N-1}(V)\in
L^{r}(\Lambda ^{m};P_{(m)}d^{m}x)$ and that $%
||K_{m,N-1}(V)||_{r,P_{(m)}d^{m}x}\leq
C_{r}(m,N-1,P_{(N-1)})||V||_{r,P_{(N-1)}d^{N-1}x}$ where $%
C_{r}(m,N-1,P_{(N-1)})>0$ is a constant independent of $V.$ That $%
K_{m,N}(U)\in L^{r}(\Lambda ^{m};P_{(m)}d^{m}x)$ and $%
||K_{m,N}(U)||_{r,P_{(m)}d^{m}x}\leq C_{r}(m,N,P)||U||_{r,Pd^{N}x}$ thus
follows from (\ref{29}) with 
\begin{multline*}
C_{r}(m,N,P):= \\
\frac{C_{r}(m,N-1,P_{(N-1)})}{N-m}\left( N\beta ^{1/r}+C_{r}(m-1,N,P)\left(
N\beta ^{1/r}+(N-m)\right) \right) .
\end{multline*}

The following corollary is the obvious variant of Corollary \ref{cor7}.
\end{proof}

\begin{corollary}
\label{cor13}Let $1\leq m\leq N$ be integers and let $1\leq r<\infty .$ If $%
N\geq 2,$ suppose that for a.e. $x_{N}$ in some subset $B\subset \Lambda $
of positive $dx$ measure, there is a constant $\gamma (x_{N})>0$ such that $%
P(\cdot ,x_{N})\geq \gamma (x_{N})P_{(N-1)}$ a.e. on $\Lambda ^{N-1}.$ Then $%
G_{m,N}$ is a linear isomorphism of $L^{r}(\Lambda ^{m};P_{(m)}d^{m}x)$ onto
the subspace of $L^{r}(\Lambda ^{N};Pd^{N}x)$ of generalized $N$-means of
order $m,$ with inverse $K_{m,N}.$ In particular, the space of generalized $%
N $-means of order $m$ in $L^{r}(\Lambda ^{N};Pd^{N}x)$ is a closed subspace
of $L^{r}(\Lambda ^{N};Pd^{N}x).$
\end{corollary}

We may once again highlight the fact that, by Corollary \ref{cor13}, the $%
L^{r}$ convergence properties of sequences of generalized $N$-means of order 
$m$ are equivalent to the same properties for the sequences of kernels. A
related albeit different issue (preservation of a \textit{product} structure
under weak convergence in $L^{r}$ spaces when $m=1$) is discussed in 
\cite[Theorem A.3]{ChChLi84}; condition (\ref{27}) is not needed.

It is obvious that (\ref{27}) holds when $Pd^{N}x$ is a product measure,
i.e. $P=\rho ^{\otimes N}$ where $\rho >0$ is a probability density on $%
\Lambda $ (and then $\gamma =\rho $). More generally, every symmetric
probability density can be approximated in $L^{1}(\Lambda ^{N};d^{N}x)$ by
positive symmetric densities satisfying (\ref{27}):

\begin{theorem}
\label{th14} If $N\geq 2$ and $P$ is a symmetric probability density on $%
\Lambda ^{N},$ there is a sequence $P_{n}>0$ of symmetric probability
densities on $\Lambda ^{N}$ such that $P_{n}$ satisfies (\ref{27}) and $%
P_{n}\rightarrow P$ in $L^{1}(\Lambda ^{N};d^{N}x).$
\end{theorem}

\begin{proof}
Since $dx$ is $\sigma $-finite, it is a simple exercise to show that there
is a probability density $\rho >0$ on $\Lambda .$ For $k\in \Bbb{N},$ set $%
Q_{k}:=\min \{P,k\rho ^{\otimes N}\}\geq 0,$ so that $Q_{k}\leq P$ and $%
Q_{k}\rightarrow P$ a.e. By dominated convergence, $Q_{k}\rightarrow P$ in $%
L^{1}(\Lambda ^{N};d^{N}x)$ and so $P_{k}:=||Q_{k}||_{1,d^{N}x}^{-1}Q_{k}%
\geq 0$ is a symmetric probability density and $P_{k}\rightarrow P$ in $%
L^{1}(\Lambda ^{N};d^{N}x).$ In addition, $P_{k}\leq
k||Q_{k}||_{1,d^{N}x}^{-1}\rho ^{\otimes N}.$ Thus, upon replacing $P$ with $%
P_{k},$ it suffices to complete the proof under the assumption $P\leq C\rho
^{\otimes N}$ for some constant $C>0.$

For $n\in \Bbb{N},$ define $P_{n}=n(n+1)^{-1}(P+n^{-1}\rho ^{\otimes N}).$
Then, $P_{n}>0$ is a symmetric probability density on $\Lambda ^{N}.$ Also, $%
P_{n}\rightarrow P$ a.e. and $P_{n}\leq P+\rho ^{\otimes N}.$ Thus, by
dominated convergence, $P_{n}\rightarrow P$ in $L^{1}(\Lambda ^{N};d^{N}x).$
We now verify (\ref{27}) for $P_{n}.$ First, $%
P_{n,(N-1)}=n(n+1)^{-1}P_{(N-1)}+(n+1)^{-1}\rho ^{\otimes (N-1)}$ and so the
assumption $P\leq C\rho ^{\otimes N}$ with $C\geq 1$ entails $%
P_{n,(N-1)}\leq (n+1)^{-1}(nC+1)\rho ^{\otimes (N-1)}.$ Next, $P_{n}\geq
(n+1)^{-1}\rho ^{\otimes N},$ so that $P_{n}$ satisfies (\ref{27}) with $%
B=\Lambda $ and $\gamma =(nC+1)^{-1}\rho .$
\end{proof}

If $\Lambda $ has finite $dx$ measure and attention is confined to
probability densities in $L^{\infty }(\Lambda ^{N};d^{N}x),$ a stronger
result is true.

\begin{theorem}
\label{th15}If $\Lambda $ has finite $dx$ measure, the set $\mathcal{P}_{*}$
of bounded symmetric probability densities $P$ on $\Lambda ^{N}$ such that $%
P>0$ and that $P$ satisfies (\ref{27}) contains an open and dense subset
(for the $L^{\infty }(\Lambda ^{N};d^{N}x)$ topology) of the set $\mathcal{P}
$ of all bounded symmetric probability densities on $\Lambda ^{N}.$
\end{theorem}

\begin{proof}
For brevity, we omit a few technical details. For $\ell \in \Bbb{N},$ call $%
\mathcal{O}_{\ell }$ the set of those $P\in \mathcal{P}$ such that $\limfunc{
ess}\sup P<\ell $ and $\limfunc{ess}\inf P>\ell ^{-1}.$ Then, $\mathcal{O}%
_{\ell }$ is open for the $L^{\infty }(\Lambda ^{N};d^{N}x)$ topology (empty
if $\ell =1$) and, if $P\in \mathcal{O}_{\ell },$ it is easily verified that 
$P>0$ satisfies (\ref{27}) with $B=\Lambda .$

If $P\in \mathcal{P}$ and $n\in \Bbb{N},$ set $Q_{n}:=\max \{P,n^{-1}\}\in
L^{\infty }(\Lambda ^{N};d^{N}x).$ Then $Q_{n}$ is symmetric, $Q_{n}>0$ and $%
Q_{n}\rightarrow P$ in $L^{\infty }(\Lambda ^{N};d^{N}x).$ Since $\Lambda $
has finite measure, $Q_{n}\rightarrow P$ in $L^{1}(\Lambda ^{N};d^{N}x)$ and
so $P_{n}:=||Q_{n}||_{1,d^{N}x}^{-1}Q_{n}\in \mathcal{P}$ and $%
P_{n}\rightarrow P$ in $L^{\infty }(\Lambda ^{N};d^{N}x).$ Furthermore, $%
P_{n}\in \mathcal{O}:=\cup _{\ell \in \Bbb{N}}\mathcal{O}_{\ell }$ (open) if 
$n>||P||_{\Lambda ^{N},\infty },$ for then $P_{n}\in \mathcal{O}_{\ell }$
with any $\ell >(n+1)\max
\{||Q_{n}||_{1,d^{N}x},||Q_{n}||_{1,d^{N}x}^{-1}\}. $
\end{proof}

In Theorem \ref{th15}, $\mathcal{P}$ is a Baire space (a closed subset of $%
L^{\infty }(\Lambda ^{N};d^{N}x)$) and $\mathcal{P}_{*}$ is residual in $%
\mathcal{P}$ in the sense of Baire category. Thus, by Baire category
estimates, the set $\mathcal{P}\backslash \mathcal{P}_{*}$ contains only
non-typical elements, but such non-typical elements may look deceptively
innocuous:

\begin{example}
\label{ex4}With $\Lambda =[0,1]$ and $dx$ the Lebesgue measure, let $N=2$
and let $P(x_{1},x_{2}):=3|x_{1}-x_{2}|.$ Then, $P>0$ a.e., $%
P_{(1)}(x_{1})=3(x_{1}^{2}-x_{1}+\frac{1}{2})$ and (\ref{27}) is equivalent
to $|x_{1}-x_{2}|\geq \gamma (x_{2})(x_{1}^{2}-x_{1}+\frac{1}{2})$ for a.e. $%
x_{1}\in [0,1],$ which does not hold for any $x_{2}\in [0,1]$ and any
constant $\gamma (x_{2})>0$ since the (continuous) left-hand side vanishes
when $x_{1}=x_{2}.$ Thus, (\ref{27}) fails.
\end{example}

If $Pd^{N}x$ is the probability that $N$ particles have the coordinates $%
x_{1},...,x_{N},$ respectively, then it is natural to require $%
P(x_{1},...,x_{N})=0$ when $x_{i}=x_{j}$ for some $i\neq j.$ A short
elaboration on Example \ref{ex4} shows that such densities $P$ do not
satisfy (\ref{27}).

\end{document}